\documentclass[11pt]{article}
\usepackage{amsmath}
\usepackage{amstext,amsbsy}
\usepackage{epsfig,multicol}
\usepackage{amsthm}
\usepackage{amssymb,latexsym}
\usepackage{amsfonts}
\usepackage{color}
\usepackage{tikz}
\usepackage{listings}
\usepackage{enumerate}
\usepackage{framed}
\usepackage{caption}
\usepackage{subcaption}
\usetikzlibrary{decorations.pathreplacing}
\usepackage{bm}

\theoremstyle{plain}
\newtheorem{theorem}{Theorem}[section]
\newtheorem{lemma}[theorem]{Lemma}
\newtheorem{corollary}[theorem]{Corollary}
\newtheorem{proposition}[theorem]{Proposition}

\theoremstyle{definition}

\newtheorem*{properties*}{Properties}

\newenvironment{definition*}[1][Definition]{\begin{trivlist}
\item[\hskip \labelsep {\bfseries #1}]}{\end{trivlist}}

\numberwithin{equation}{section}

\newcommand{\R}{\mathcal{R}}

\DeclareMathOperator{\maj}{maj}
\DeclareMathOperator{\inv}{inv}
\DeclareMathOperator{\des}{des}
\DeclareMathOperator{\Des}{Des}
\DeclareMathOperator{\sor}{sor}

\DeclareMathOperator{\bcode}{B-code}

\DeclareMathOperator{\maxi}{i}
\DeclareMathOperator{\maxs}{s}

\newcommand{\majup}{\maj'_U}
\newcommand{\invup}{\inv'_{U}}
\newcommand{\desup}{\des'_{U}}
\newcommand{\Desup}{\Des'_{U}}
\newcommand{\sorup}{\sor'_{U}}

\newcommand{\al}{\boldsymbol{\alpha}}
\newcommand{\Ra}{\mathcal{R}(\al)}

\begin{document}
\title{Graphical Mahonian Statistics on Words}
\author{
Amy Grady and Svetlana Poznanovi\'c$^1$  \\ [6pt]
Department of Mathematical Sciences\\
Clemson University, Clemson, SC 29634  \\
}
\date{}
\maketitle
\begin{abstract} 
Foata and Zeilberger defined the graphical major index, $\majup$, and the graphical inversion index, $\invup$, for words.  These statistics are a generalization of the classical permutation statistics $\maj$ and $\inv$ indexed by directed graphs $U$.  They showed that $\majup$ and $\invup$ are equidistributed over all rearrangement classes if and only if  $U$ is bipartitional. In this paper we strengthen their result by showing that if $\majup$ and $\invup$ are equidistributed on a single rearrangement class then $U$ is essentially bipartitional. Moreover, we define a graphical sorting index, $\sorup$, which generalizes the sorting index of a permutation. We then characterize the graphs $U$ for which $\sorup$ is equidistributed with $\invup$ and $\majup$ on a single rearrangement class.

\end{abstract}

{\renewcommand{\thefootnote}{} \footnote{\emph{E-mail addresses}:
agrady@clemson.edu (A. Grady), spoznan@clemson.edu (S.~Poznanovi\'c)}}

\footnotetext[1]{The second author is partially supported by the NSF grant DMS-1312817.}


\section{Introduction} \label{S:introduction}

Let $\al = (\alpha_1, \alpha_2, \ldots, \alpha_n)$ be a sequence of nonnegative integers. We will denote by $\R(\al)$ the set of permutations of the multiset $\{1^{\alpha_{1}}, 2^{\alpha_{2}}, \ldots, n^{\alpha_{n}} \}$, i.e., $\R(\al)$ is the set of all words containing $\alpha_i$ occurrences of the letter $i$ for all $i=1,2,\ldots,n$. For $w = x_1 x_2 \ldots x_m \in \R(\al)$, the \emph{inversion number} is defined as 
\[ \inv w = \sum _{1\leq i < j \leq m} \mathcal{X} (x_i > x_j),\]
and the \emph{major index} is defined as
\[  \maj w = \sum _{i=1} ^{m-1} i\mathcal{X} (x_i > x_{i+1} ).\] The set of all positions $i$ such that $x_i > x_{i+1}$ is known as the \emph{descent set} of $w$, $\Des w$, and its cardinality is denoted by $\des w$. So, $\maj w = \sum_{i \in \Des w} i$. 

The generating function for permutations by number of inversions goes back to Rodriguez~\cite{rodriguez1839note} and the generalization to multisets is due to MacMahon~\cite{macmahon1917two}. MacMahon also showed~\cite{macmahon1913indices, macmahon1984combinatory} that $\maj$ and $\inv$ are equidistributed on $\R(\al)$. Namely,
\[  \sum _{w \in \R(\al)} q^{\inv w} = \sum _{w \in \R(\al)} q^{\maj w} =\left[\begin{array}{c}
\alpha_1 + \alpha_2+ \cdots+ \alpha_n \\
\alpha_1, \alpha_2, \ldots, \alpha_n \\
\end{array}\right] \] 
where 

\[ \left[\begin{array}{c}
\alpha_1 + \alpha_2+ \ldots+ \alpha_k \\
\alpha_1, \alpha_2, \ldots, \alpha_k \\
\end{array}\right] = \frac {[\alpha_1 + \alpha_2+ \cdots+ \alpha_k ]!}{[\alpha_1]! [\alpha_2]! \dots [\alpha_k]!}\]
is the $q$-analog of the multinomial coefficient and $[n]! =  (1 + q) (1 + q + q^2) \cdots (1+ q + q^2 + \dots + q^{n-1})$ is the $q$-factorial.

In honor of MacMahon, all permutation statistics that share the same distribution are called Mahonian. These two classical Mahonian statistics have been generalized in various ways. Some examples are Kadell's weighted inversion number~\cite{kadell1985weighted}, the $r$-major index introduced by Rawlings~\cite{rawlings1981r},  the statistics introduced by Clarke~\cite{clarke2005note}, and the maj-inv statistics of Kasraoui~\cite{kasraoui2009classification}. The generalization that we will be considering in this paper is due to Foata and Zeilberger~\cite{foata1996graphical}. They defined graphical  statistics (graphical inversions and graphical major index) parameterized by a general directed graph $U$ and they described the graphs $U$ for which these statistics are equidistributed on \emph{all} rearrangement classes.

\begin{theorem}[\cite{foata1996graphical}] \label{fzthm}
The statistics $\invup$ and $\majup$ are equidistributed on each rearrangement class $\R(\al)$ if and only if the relation $U$ is bipartitional.
\end{theorem}
A similar result was proved in~\cite{foata1995graphical}, where the definition of graphical inversions and major index is modified to allow different behavior of the letters at the end of the word. 

Here we do two different things. First, we strengthen Foata and Zeilberger's result by showing that the equidistribution of $\invup$ and $\majup$ on a \emph{single} rearrangement class $\R(\al)$ implies that $U$ is essentially bipartitional (Theorem~\ref{thm: main inv}). Second, we define a graphical sorting index on words, a statistics which generalizes the sorting index for permutations~\cite{petersen2011sorting}. We then describe the directed graphs $U$ for which $\sorup$ is equidistributed with $\invup$ and $\majup$ on a fixed class $\R(\al)$ (Theorem~\ref{thm: main sor}). 

In the next section we define the terminology we need and state the main results. Then we prove Theorem~\ref{thm: main inv} and Theorem~\ref{thm: main sor} in Section~\ref{sec: inv and maj} and Section~\ref{sec: sor}, respectively.


\section{Preliminaries and Main Results} 

A {\it directed graph} on $X=\{1,2, \ldots, n\}$ is any subset $U$ of the Cartesian product $X \times X$.  For each such directed graph $U$, we have the following statistics defined on each word $w = x_1 x_2 \ldots x_m$ with letters from $X$:
\begin{align*}
\invup w &= \sum_{1 \leq i < j \leq m} \mathcal{X} ((x_i, x_j) \in U),\\
\Desup w &=\{i \colon 1 \leq i \leq m,  (x_i, x_{i+1}) \in U\},\\
\desup w &=|\Desup|,\\
\majup w &= \sum_{i \in \Desup w} i. 
\end{align*}
Since $U$ is also a relation on $X$, for convenience, in some places we will use the notation $x>_{U}y$ to represent the edge $(x,y) \in U$. We will say $x$ is related to $y$ if $(x,y) \in U$ or $(y,x) \in U$. 

An {\it ordered bipartition} of $X$ is a sequence $(B_1, B_2, \ldots, B_k)$ of nonempty disjoint subsets of $X$ such that $B_1 \cup B_2 \cup \dots \cup B_k = X$, together with a sequence $(\beta_1, \beta_2, \dots, \beta_k)$ of elements equal to 0 or 1.  If $\beta_i = 0$ we say the subset $B_i$ is {\it non-underlined}, and if $\beta_i = 1$ we say the subset $B_i$ is {\it underlined}.

A relation $U$ on $X \times X$ is said to be {\it bipartitional}, if there exists an ordered bipartition $((B_1, B_2, \ldots, B_k), (\beta_1, \beta_2, \ldots, \beta_k))$ such that $(x,y) \in U$ if and only if either $x \in B_i$, $y\in B_j$ and $i<j$, or $x$ and $y$ belong to the same underlined block $B_i$. Bipartitional relations were introduced in~\cite{foata1996graphical} as an answer to the question ``When are $\invup$ and $\majup$ equidistributed over all rearrangement classes?''. In particular, there the authors showed that if $U$ is bipartitional with blocks $((B_{1}, \ldots, B_{k}), (\beta_{1},\ldots, \beta_{k}))$ then
\begin{equation}\label{fzgf} \sum_{w \in \R(\al)} q^{\invup w} = \sum_{w \in \R(\al)} q^{\majup w} =  {|\alpha| \brack m_1,\ldots, m_k} \prod_{j=1}^k {m_l \choose \alpha(B_l)}q^{\beta_{j} \binom{m_{j}}{2}}.\end{equation} Here and later we use the notation
\begin{align*} |\al| &= \alpha_{1}+ \cdots + \alpha_{n},\\
m_{i} &= |B_{i}|,\\
\al(B_{i}) &= (\alpha_{i_{1}}, \ldots, \alpha_{i_{l}}) \text{ if } B_{i} = \{i_{1} < \cdots < i_{l}\}.\end{align*}
Han \cite{han1995ordres} showed that bipartitional relations $U$ can also be characterized as relations $U$ for which both $U$ and its complement are transitive. Hetyei and Krattenthaler~\cite{hetyei2011poset} showed that the poset of bipartitional relations ordered by inclusions has nice combinatorial properties.

In this paper we will be considering the distribution of $\invup$ and $\majup$ over a fixed rearrangement class $\R(\al)$. Notice that if the multiplicity $\alpha_{x}$ of $x \in X$ is 1, then the pair $(x,x)$ cannot contribute to neither $\invup$ nor $\majup$. Therefore, omitting or adding such pairs to $U$ doesn't change these two statistics over $\R(\al)$. For that purpose, we define $U$ to be \emph{essentially bipartitional relative to $\al$} if there are disjoint sets $I \subseteq X$ and $J \subseteq X$ such that 
\begin{enumerate}[(1)]
\item $\alpha_{x}=1$ for all $x \in I \cup J$ and 
\item $(U \setminus \{(x,x) \colon x \in I\}) \cup  \{(x,x) \colon x \in J\}$ is bipartitional.
\end{enumerate}

\begin{theorem} \label{thm: main inv}
The statistics $\invup$ and $\majup$ are \emph{equi\textbf{di}stributed} over $\R(\al)$ if and only if the relation $U$ is essentially bipartitional relative to $\alpha$.
\end{theorem}

In view of the comment preceding the theorem, the``if" part of Theorem \ref{thm: main inv}  follows from Theorem~\ref{fzthm}. We prove the ``only if" in Section \ref{sec: inv and maj}.

The third Mahonian statistic we will consider is the \emph{sorting index} introduced by Peterson \cite{petersen2011sorting} and also studied independently by Wilson~\cite{wilson2010interesting}. Every permutation $\sigma \in S_n$ can uniquely be  decomposed as a product of transpositions,  $\sigma = (i_1,j_1)(i_2,j_2) \cdots (i_k,j_k)$, such that $j_1<j_2<\cdots <j_k$ and $i_1<j_1, i_2<j_2,  \ldots , i_k<j_k$.  The sorting index is defined by \[\sor \sigma = \sum_{r=1}^k (j_r-i_r).\]  The desired transposition decomposition can be found using the Straight Selection Sort algorithm. The algorithm first places $n$ in the $n$-th position by applying a transposition, then places $n-1$ in the $(n-1)$-st position by applying a transposition, etc. For example, for  $\sigma = 2413576$, we have \[2413576 \overset{(67)} \rightarrow 2413567 \overset{(24)} \rightarrow 2314567 \overset{(23)} \rightarrow 2134567 \overset{(12)} \rightarrow 1234567 \] and, therefore, $\sor \sigma = (2-1)+(3-2)+(4-2)+(7-6) = 5$. 

The sorting index has been extended to labeled forests by the authors~\cite{grady2015sorting}. It can also be naturally extended to words $w \in \R(\al)$ by a generalization of Straight Selection Sort which reorders the letters into a weakly increasing sequence. At each step transpositions are applied to place all the $n$'s at the end, then all the $n-1$'s to the left of them, etc, so that for each $x \in X$,  the $\alpha_{x}$ copies of $x$ stay in the same relative order they were right before they were ``processed''. Then we define $\sor w$ to be the sum of the number of positions each element moved during the sorting. For example, applying this sorting  algorithm to $w= 143123123$ yields \begin{equation} \label{sorting} 143123123  \rightarrow 133123124 \rightarrow 123123134  \rightarrow 123121334 \rightarrow 121123334 \rightarrow 111223334 \end{equation} and thus $\sor w = 7 + 2 +3 +4 + 2 = 18$.

We define a graphical sorting index that depends on $U$ using the same sorting algorithm but at each step, when sorting $x$, we only count how many elements $y$ such that $(x,y) \in U$ it ``jumps over''. More formally, to compute $\sorup w$ for $w=x_{1}x_{2}\ldots x_{m}$: 
\begin{itemize}
\item Begin with $i=m$, and $\sorup w = 0$.
\item Consider the largest element in $w$ with respect to integer order.  If there is a tie, pick the element with the largest subscript, and call this element $x_j$.
\item Interchange $x_j$ with $x_i$. 
\item For each $h=j+1,j+2,\ldots,i$, if $(x_j, x_h) \in U$ increase $\sorup w$ by 1.
\item Repeat this process for $i = m-1, \ldots, 1$.
\end{itemize}

For example, consider the same word $w= 143123123$  with $U = \{ (4,3), (3,3), (3,1), (2,3), (1,1)\}.$  The sorting steps are given in~\eqref{sorting} and thus $\sorup w = 3 + 1 +1 +2 + 0= 7$. In particular, if $U$ is the natural integer order $>$ then $\sor w = \sorup w$. Our second main result is the following.

\begin{theorem}\label{thm: main sor} 
The statistics $\sorup$, $\invup$ and $\majup$ are equidistributed on a fixed rearrangement class $\R(\al)$ if and only if the relation $U$ has the following properties.
\begin{enumerate}
\item $U$ is bipartitional with no underlined blocks, 
\item If $(x, y) \in U$ then  $x > y$, 
\item All but the last block of $U$  are of size at most 2,
\item If $U$ has $k$ blocks $B_{1}, \ldots, B_{k}$ and $|B_{i}|=2$ for some $1 \leq i \leq k-1$ then $\alpha_{\max B_{i}}=1$.
\end{enumerate}
\end{theorem}

We give the proof of Theorem \ref{thm: main sor} in Section \ref{sec: sor}. We mention also that recently the question ``When are $\sor$ and $\inv$  equidistributed on order ideals of the Bruhat order?'' was recently addressed in~\cite{fan2016sorting}.

\section{The Proof of Theorem~\ref{thm: main inv}} \label{sec: inv and maj}

We begin the proof with a simple observation.
\begin{lemma} \label{complement}
The statistics $\majup$ and $\invup$ are equidistributed on $\R(\al)$ if and only if $\maj'_{U^{c}}$ and $\inv'_{U^{c}}$ are equidistributed on $\R(\al)$.
\end{lemma}
\begin{proof}
This follows from the fact that for every $w \in \R(\al)$, \[\majup w + \maj'_{U^{c}} w = \binom{|\al|}{2} = \invup w + \inv'_{U^{c}} w.\]

\end{proof}

\begin{lemma} \label{lem: inv < maj}
For any $\al = (\alpha_{1}, \ldots, \alpha_{n})$ and any relation $U$ on $X = \{1, 2, \ldots, n\}$,
\[ \max_{w \in \Ra} \majup{w} \geq \max_{w \in \Ra} \invup{w}.\]
\end{lemma}

\begin{proof}
We will use induction on $|\al|$. It's clear that the statement holds when $|\al| =1$.  Assume that it holds for all $\al$ with $|\al| \leq m$.

Consider a rearrangement class $\Ra$ such that $|\al| = m+1$ and a relation $U$ on $[n]$.  Let $(\al, U)$ be a directed graph with vertex set $\{1^{\alpha_{1}}, \ldots, n^{\alpha_{n}}\}$ and a directed edge $x \rightarrow y$ whenever $(x,y) \in U$. Let $x_{1} \rightarrow x_{2} \rightarrow \cdots \rightarrow x_{n}$ be a directed path in $(\al, U)$ of maximal possible length. This means we have a descending chain $x_1>_{U}x_2>_{U}\cdots >_{U} x_l$  of maximal possible length. Set $\al'=(\alpha'_{1}, \ldots, \alpha'_{n})$ where 
\[ \alpha'_i = \alpha_i - \sum _{j=1} ^l \mathcal{X} (x_j = i).\]
Let $u'$ be a word that maximizes  $\majup$ on the rearrangement class $\mathcal{R}(\al')$. One can easily verify that for the word $u = u' x_1 x_2 \cdots x_l$ in $\Ra$ we have \begin{equation} \label{eq: maju} \majup u = \majup u' + \frac {(l-1)(2m+2-l)}{2}.\end{equation}

To bound $\max_{w \in \Ra} \invup{w}$, first suppose there is an element $y \in (\al',U)$ is such that for all $i=1,2,\ldots,l$ we have $(y,x_{i}) \in U$ or $(x_{i},y) \in U$. If $(y, x_{1}) \in U$ then $y>_{U}x_1>_{U}x_2>_{U}\dots>_{U}x_l$ is a longer chain in $(\al, U)$, therefore $(y, x_1)\notin U$ and $(x_1,y)\in U$. Similarly, if $(x_l, y)\in U$ we can form the longer chain $x_1>_{U}x_2>_{U}\dots>_{U}x_l>_{U}y$ in $(\al, U)$; thus we must have $(x_l,y)\notin U$ and $(y,x_l)\in U$. However, this implies that there are elements $x_i$ and $x_{i+1}$ such that $(x_i,y), (y,x_{i+1}) \in U$, which yields a longer chain $x_1>_{U}x_2>_{U}\dots>_{U}x_i>_{U}y>_{U}x_{i+1}>_{U}\dots>_{U}x_l$. Therefore, every  $y \in (\al',U)$ is related to at most $l-1$ elements in the chain $x_{1} >_{U} \cdots >_{U} x_{l}$. 

Now consider a word $v \in \Ra$ and the corresponding word $v' \in \mathcal{R}(\al')$ obtained by deleting $x_{1}, \ldots, x_{l}$.  By the argument in the previous paragraph, the $m+1-l$ letters in $v'$ create at most $(m+1-l)(l-1)$   graphical inversions with $x_{1}, \ldots, x_{l}$. Therefore, by~\eqref{eq: maju} and the induction hypothesis,

\begin{align} \max_{w \in \Ra} \invup{w} &\leq \max_{w' \in \mathcal{R}(\al')} \invup{w} + (m+1-l)(l-1) + \binom{l}{2} \label{lm2.1} \\ &= \max_{w' \in \mathcal{R}(\al')} \invup{w} + \frac {(l-1)(2m+2-l)}{2} \label{lm2.2} \\ &\leq \max_{w' \in \mathcal{R}(\al')} \majup{w} + \frac {(l-1)(2m+2-l)}{2} \label{lm2.3} \\ &\leq \max_{w \in \Ra} \majup{w}  \label{lm2.4}.\end{align}
\end{proof}

The proof of Lemma~\ref{lem: inv < maj} also shows that a word $w=w_{k}w_{k-1}\cdots w_{1}$ with the property $\majup(w) \geq \max_{v \in \Ra} \invup{v}$ can be constructed by ``peeling off'' descending chains of maximal length from $(\al, U)$ and ordering them from right to left, forming the subwords $w_{1}, w_{2}, \ldots, w_{k}$ in that order. These kind of words will be used in the proof and  for a fixed relation $U$, we will call such words \emph{maximal chain words} in $\R(\al)$.

\begin{lemma} \label{lem:chain prop}
Suppose $\majup$ and $\invup$  are equidistributed on $\Ra$. Let $w = w_k w_{k-1}\cdots w_1 \in \Ra$ be a maximal chain formed from the maximal chains $w_{1}, w_{2}, \ldots, w_{k}$. Then
\begin{enumerate}[(i)]
\item \label{it1} For each of the maximal descending chains $w_j = x_{i_{j-1}+1} x_{i_{j-1}+2} \cdots x_{i_j}$
 \begin{equation} (x_{r}, x_{s}) \in U \; \text{ or  } \; (x_{s}, x_{r}) \in U \; \text{ for all } \; i_{j-1}+1 \leq r < s \leq i_{j},\end{equation}
 \item \label{it2} Each letter $y$ in a maximal descending chain $w_{i}, i > j$, is in relation with exactly $i_{j}-1$ elements from $w_{j}$, i.e., there is a unique $r \in \{i_{j-1}+1, \ldots, i_{j}\}$ such that $(y, x_{r}) \notin U$ and  $(x_{r},y) \notin U$. Moreover, $(x_{s}, y) \in U$ for $i_{j-1}+1 \leq s < r$ and $(y, x_{s}) \in U$ for $r < s \leq i_{j}$. 
 \end{enumerate}

\end{lemma}

\begin{proof}
Condition $(i)$ is necessary for equality to hold in~\eqref{lm2.1}. The property $(ii)$ also follows from the fact that equality  holds in~\eqref{lm2.1} and the definition of a maximal chain word which implies that the chain $w_{j}$ is be the longest one that can be formed among the letters in $w_{k} w_{k-1}\cdots w_{j}$. \end{proof}

The following lemma shows that if $\majup$ and $\invup$ are equidistributed  on $\Ra$ the elements in the maximal chains can be reordered, if necessary, so that within each of them the following property holds: if $x$ precedes $y$ in the same chain of a maximal chain word then $(x,y) \in U$.

\begin{lemma} \label{lem: trans}
If $\majup$ and $\invup$  are equidistributed on $\Ra$, then there exists a maximal chain word $w = w_k w_{k-1} \cdots w_1 \in \Ra$ with subwords $w_i$ formed from descending chains such that for any  $w_j = x_{i_{j-1}+1} x_{i_{j-1}+2} \cdots x_{i_j}$ we have \begin{equation} \label{transitivity} (x_{r}, x_{s}) \in U \; \text{ for all } \; i_{j-1}+1 \leq r < s \leq i_{j}.\end{equation}
\end{lemma}

\begin{proof}
Since the equality in~\eqref{lm2.1} holds, the elements $x_{1},x_{2}, \ldots, x_{l}$ in the maximal chain can be  arranged so that they form $\binom{l}{2}$ graphical inversions, which implies the statement in the lemma.
\end{proof}

\begin{lemma} \label{lem: same chain}
Suppose $\majup$ and $\invup$ are equidistributed on $\R(\al)$.  Let $w = w_{k}w_{k-1}\cdots w_{1}$ be a  maximal chain word in $\R(\al)$ for $U$ with maximal chains $w_{1}, \ldots, w_{k}$. If $(x,y) \in U$ and $(y,x) \in U$ for some $x \neq y$, then the $x$'s and $y$'s are all in the same chain $w_{i}$.
\end{lemma}

\begin{proof}
Without loss of generality, suppose there is an $x$ that appears in a chain $w_{j_{1}}$ and a $y$ that appears in the chain $w_{j_{2}}$, $j_{1} > j_{2}$. Consider the chain $w_{j_{2}}: b_1 >_{U} b_2 >_{U} \dots b_{l-1} >_{U} y >_{U}b_{l+1} >_{U} \dots >_{U} b_{m}$. By Lemma~\ref{lem:chain prop}, there is exactly one $i \in \{1, 2, \ldots, m\}$ such that  $(x, b_{i}), (b_{i},x) \notin U$, $(b_{1},x), \ldots, (b_{i-1},x) \in U$, $(x, b_{i+1}), \ldots, (x, b_{m}) \in U$.   If $l< i$ then the chain $b_1>_{U}b_2>_{U}\dots>b_{l-1}>_{U}x>_{U}y>_{U}b_{l+1}>_{U}\dots>_{U}b_m$ is a longer chain  than $w_{j_{2}}$ and if $l>i$ then $b_1>_{U}b_2>_{U}\dots>_{U}b_{l-1}>_{U}y>_{U}x>_{U}b_{l+1}>_{U}\dots>_{U}b_m$ is a longer chain than $w_{j_{2}}$. This contradicts the definition of a maximal chain word.

\end{proof}

\begin{lemma}
Suppose there exists an element $x \in X$ with $\alpha_{x} \geq 1$ such that $U \subset (X \setminus \{x\}) \times (X \setminus \{x\})$ and $U \neq \emptyset$. Then  $\invup$ and $\majup$ are not equidistributed over $\R(\al)$. 
\end{lemma}

\begin{proof}
Set $$\alpha'_{i} = \begin{cases}
\alpha_{i} \; & \text{ if } i\neq x,\\
0 \; &  \text{ if } i=x.
\end{cases}$$ 
Let $w' \in \R(\al')$ be the word such that $\majup{w'} = \max_{u \in \R(\al')} \majup{u}$. Consider the word $w = \underbrace{xx\cdots x}_{\alpha_{x}} w'$ in $\R(\al)$. Since $x$ does not create any graphical inversions, by Lemma~\ref{lem: inv < maj}, we have
\begin{align*} \max_{u \in \R(\al)} \majup{u} \geq \majup{w} &= \majup{w'} + \alpha_{x}\des'_{U}{w'} \\& = \max_{u \in \R(\al')} \majup{u} + \alpha_{x}\des'_{U}{w'} \\ & \geq \max_{u \in \R(\al')} \invup{u} + \alpha_{x}\des'_{U}{w'} \\
& = \max_{u \in \R(\al)} \invup{u} + \alpha_{x}\des'_{U}{w'} \end{align*}
Consequently, if $\invup$ and $\majup$ are equidistributed on $\R(\al)$, then $\des'_{U}{w'} =0$ and thus $$\max_{u \in \R(\al')} \maj u =0.$$ This contradicts the fact that $U \neq \emptyset$.

\end{proof}

\begin{lemma} \label{lem: x,y underlined}
Suppose $\majup$ and $\invup$ are equidistributed on  $\R(\al)$.  If $(x,y), (y,x) \in U$ and $\alpha_{x}>1$ then $(x,x) \in U$.\end{lemma}

\begin{proof}
Since $(x,y), (y,x) \in U$, by Lemma~\ref{lem: same chain}, all the $x$'s and $y$'s must be in the same maximal chain of a maximal chain word. In particular, since two $x$'s are in the same chain, part $(i)$ of Lemma~\ref{lem:chain prop} implies that $(x,x) \in U$.

\end{proof}

\begin{lemma} \label{lem: underlined blocks 13}

Suppose $\majup$ and $\invup$ are equidistributed over $\R(\al)$ and let $x$ and $y$ be two distinct elements of $X$ such that $(x,y), (y,x) \in U$. For every $z \in \{1^{\alpha_{1}}, \ldots, n^{\alpha_{n}} \}\setminus \{x,y\}$, we have
\[(z,x)\in U \text{ if and only if }(z,y)\in U\] 
\[\text{and }\]
\[(x,z)\in U \text{ if and only if }(y,z)\in U.\]
\end{lemma}

\begin{proof} If $z=x$ then $\alpha_{x} >1$ and the claim follows from Lemma~\ref{lem: x,y underlined}. The same is true if $z=y$. So, suppose $z \neq x$, $z \neq y$. Because of symmetry, it suffices to prove 
\begin{align}
(z,x) \in U \implies (z,y) \in U \label{first}\\
(x,z) \in U \implies (y,z) \in U \label{second}
\end{align}
To see~\eqref{first}, suppose that $(z,x) \in U, (z,y) \notin U$. We consider two cases. 

\underline{\textbf{Case 1:}} $(y,z) \notin U$. Let $w = w_{t}w_{t-1}\cdots w_{1} \in \R(\al)$ be a maximal chain word that satisfies~\eqref{transitivity}. By Lemma~\ref{lem: same chain}, $x$ and $y$ are in the same chain $w_{i}$ of $w$. By Lemma~\ref{lem:chain prop}, $z$ is in a different chain $w_{j}$ and by Lemma~\ref{lem: same chain}, $(x,z) \notin U$. If $j > i$, notice that, by Lemma~\ref{lem:chain prop}, $x$ cannot precede $y$ in $w_{i}$, so $w_{i}$ must be of the form $w_{i}=b_{1}\cdots b_{k}y b_{k+1}\cdots b_{l}x b_{l+1}\cdots b_{m}$. Then $b_{1}\cdots b_{k}z b_{k+1}\cdots b_{l}xy b_{l+1}\cdots b_{m}$ is a descending chain longer than $w_{i}$. If $j <i$, then $w_{j} = b_{1}\cdots b_{k}z b_{k+1}\cdots b_{l}$. By part $(ii)$ of Lemma~\ref{lem:chain prop}, $(b_{k},x), (y,b_{k+1}) \in U$, which implies that $b_{1}\cdots b_{k} x y b_{k+1}\cdots b_{l}$ is a descending chain longer than $w_{j}$.

\underline{\textbf{Case 2:}} $(y,z) \in U$. By Lemma~\ref{complement}, $\maj'_{U^{c}}$ and $\inv'_{U^{c}}$ are equidistributed on $\R(\al)$. Let $w= w_{t}w_{t-1}\cdots w_{1} \in \R(\al)$ be a maximal chain word for $U^{c}$ that satisfies~\eqref{transitivity}. Suppose $x, y, z$ are in the chains $w_{i}, w_{j},w_{k}$, respectively. By Lemma~\ref{lem:chain prop}, $i \neq j$, $i \neq k$. If $i < j,k$ and $w_{i} = b_{1} \cdots b_{l} x b_{l+1} \cdots b_{m}$ then a different maximal chain word $w'$ could be constructed by taking the same chains $w_{1}, \ldots, w_{i-1}$ as in $w$  and replacing $w_{i}$ by $b_{1} \cdots b_{l} y b_{l+1} \cdots b_{m}$. Since $(z,y)\in U^{c}$, it follows from Lemma~\ref{lem:chain prop} that $z$ is not in relation $U^{c}$ with some $b_{r}$, $r \leq l$ and therefore $(z,x) \in U^{c}$, which contradicts $(z,x) \in U$. The similar argument holds if $j <i,k$.  If $k < i,j$ and $w_{k} = b_{1} \cdots b_{l} z b_{l+1} \cdots b_{m}$ then $y$ is not in relation $U^{c}$ with some $b_{r}$, $r > l$, and a different maximal chain word for $U^{c}$ could be formed by replacing $w_{k}$ with $b_{1} \cdots b_{l} z b_{l+1} \cdots b_{r-1} y b_{r+1} \cdots b_{m}$. Part $(ii)$ of Lemma~\ref{lem:chain prop} now implies that $(z,x) \in U^{c}$, which contradicts $(z,x) \in U$. Finally, if $j=k <i$, then since $(z,x) \notin U^{c}$ and $(x,y),(y,x) \notin U^{c}$, Lemma~\ref{lem:chain prop} implies that $(x,z) \in U^{c}$ and $y$ precedes $z$ in $w_{j}$. Therefore, $(y,z) \in U^{c}$, which contradicts $(y,z) \notin U$. 

The implication~\eqref{second} can be proved by considering completely analogous cases, so we omit it here. 
\end{proof}

For a relation $U$ on $X$, call $S(U) = \{ (x,y) \in X \times X : (x,y), (y,x) \in U\}$, the {\it symmetric} part of $U$, and call $A(U) = U \setminus S(U)$  the {\it asymmetric} part of $U$.   Let $X_U = \{ x \in X : (x,y) \in S(U)$ for some $y \in X \}$.  

 \begin{lemma} \label{lem: eq rel}
If $\majup$ and $\invup$ are equidistributed over a rearrangement class $\R(\al)$, then $S(U) \cup \{(x,x): x \in X_{U}, \alpha_{x}=1\}$ is an equivalence relation on $X_U \times X_U$.
\end{lemma}

\begin{proof}
Let $x\in X_U$ and $y\in X$ such that $(x,y) \in U$.  If $y = x$ then we have $(x,x) \in U$.  If $y\neq x$ then we have $(x,y),(y,x) \in U$ and thus, if $\alpha_{x}>1$ by Lemma \ref{lem: x,y underlined} we have $(x,x)\in U$. $S(U)$ is symmetric by definition because $(x,y) \in S(U)$ implies $(y,x) \in S(U)$. Now consider $x,y,z \in X_U$ and assume $(x,y), (y,z) \in S(U)$.  Then by definition of $S(U)$, $(y,x), (z,y) \in S(U)$   and Lemma \ref{lem: underlined blocks 13} implies $(x,z) \in S(U)$.
\end{proof}

Consequently, $X_{U}$ can be partitioned into blocks $B_1, \ldots, B_l$ such that \[S(U) \cup \{(x,x): x \in X_{U}, \alpha_{x}=1\} = (B_1\times B_1) \cup (B_2 \times B_2) \cup \dots \cup (B_l \times B_l).\]

\begin{lemma} \label{lem: smallest block}
Suppose that $\majup$ and $\invup$ are equidistributed over a rearrangement class $\R(\al)$.  Then either 
there is a block $B$ of $X_{U}$ such that 
\[ \text{ for all }x \in B\text{ and all } y\in X \setminus B\text{ we have } (x,y) \notin U,\]
or there is an element 
\[x\in X \setminus X_U\text{ such that for all }y\in X \backslash \{x\} \text { we have }(x,y) \notin U.\]
\end{lemma}

\begin{proof}
Suppose that the lemma does not hold.  In other words, assume that $\majup$ and $\invup$ are equidistributed over a rearrangement class $\R(\al)$, but for all blocks $B_i$ of $X_U$ there exists a $x\in B_i$ and $y \in X \setminus B_i$ such that $(x,y) \in U$, and for all $x \in X \setminus X_U$ there exists a $y\in X \setminus \{x\}$ such that $(x,y) \in U$.

Consider $x_0 \in X$.  If $x_{0} \in B_{i_0}$ for some $B_{i_0} \subset X_U$ there exists a $x_1 \in B_{i_0}$ and $x_2 \in X\setminus B_{i_0}$ such that $(x_0 ,x_1),(x_1,x_0) \in U$, and $(x_1,x_2) \in U$. Lemma \ref{lem: underlined blocks 13} implies $(x_0,x_2) \in U$, and $(x_2,x_0) \notin U$ because if so $x_2 \in B_{i_0}$. Note that if we began with $x_0 \notin X_U$ our assumptions would still give an element $x_2 \in X\setminus \{x\}$ such that $(x_0,x_2) \in U$, and $(x_2,x_0) \notin U$ because $x_{0} \notin X_U$. Now there are two cases to consider.

\underline{\textbf{Case 1:}} $x_2 \in B_{i_1}$ for some $B_{i_1} \subset X_U$ and $B_{i_1} \neq B_{i_0}$.  Then there exists a $x_3 \in B_{i_1}$ and $x_4 \notin B_{i_1}$ such that $(x_{4},x_3) \in U$  and, by Lemma \ref{lem: underlined blocks 13}, $(x_2, x_{4}) \in U$ and $(x_4, x_2) \notin U$.

\underline{\textbf{Case 2:}} $x_{2} \notin X_U$, and then there exists a $x_4 \in X \setminus \{x_2\}$ such that $(x_2, x_{4}) \in U$ and $(x_4, x_2) \notin U$.

Continuing this process we can build a sequence $x_0, x_2, x_4, x_6, \ldots$ with the properties $x_{0} >_{U} x_{2} >_{U} x_{4} >_{U} x_{6} >_{U} \cdots$ and $x_{0} \not<_{U} x_{2} \not<_{U} x_{4} \not<_{U} x_{6} \not<_{U} \cdots$   with $x_{2i}  \neq x_{2i+2}$ for all $i$.

The set $X$ is finite and thus this sequence can not be infinite with distinct terms.  Therefore, with relabeling there is a finite sequence $y_1, y_2, y_3, \ldots, y_{l+1}$ such that
\begin{itemize}
\item $l \geq 2$
\item all $y_i$'s are distinct for $i=1,2,\ldots,l$
\item $y_{1}>_{U} y_{2} >_{U} (y_{3} >_{U} y_{4} >_{U} \cdots >_{U} y_{l} >_{U}y_{l+1} = y_{1} $
\item $y_{1} \not<_{U} y_{2} \not<_{U} y_{3} \not<_{U} y_{4} \not<_{U} \cdots \not<_{U} y_{l} \not<_{U} y_{l+1} = y_{1}$
\end{itemize}

If $l =2$ then we have $y_{1} <_{U} y_{2} <_{U} y_{1} $ and $y_{1} \not<_{U} y_{2} \not< _{U}y_{1} $  which is a contradiction and hence $l \geq 3$.

Let $w \in \R(\al)$ be a maximal chain word for $U$. If all $y_{1}, \ldots, y_{l}$ appear in the same chain $w_{i}$, then by Lemma~\ref{lem: trans}, they can be relabeled to give  a sequence $z_{1}, \ldots, z_{l}$ such that $(z_{r}, z_{s}) \in U$ for all $1 \leq r < s \leq l$. If $z_{l} = y_{i}$, then this means that $y_{i+1}>_{U} y_{i} $, which is a contradiction. If not all all $y_{1}, \ldots, y_{l}$ appear in the same chain let $y_{j}$ be the one that appears in the rightmost chain of $w$. Then either $y_{j+1}$ is already in the same chain or its not related to an element to the right of $y_{j}$. In the latter case, another maximal chain word can be constructed in which $y_{j}$ and $y_{j+1}$ are in the same chain, while the other $y_{i}$'s are either in the same chain or in chains to the left. Continuing this argument, we see that we can construct a maximal chain word in which all $y_{1}, \ldots, y_{l}$ are in the same chain, which as we saw before is impossible.

\end{proof}

Let $C = \{ x\in X\setminus X_U \colon \alpha_{x}>1, (x,y) \notin U,  \forall y \in X \text{ or } \alpha_{x}=1, (x,y) \notin U,  \forall y \in X \setminus \{x\}\}.$

\begin{lemma} \label{lem: block structure}
Suppose that $\majup$ and $\invup$ are equidistributed over a rearrangement class $\R(\al)$.  If $C$ is nonempty then 
\[ \text{for all }y \in X \setminus C \text{ and for all }x \in C \text{ we have }(y,x) \in U.\]
If $C$ is empty and $B$ is the block defined in Lemma \ref{lem: smallest block} then 
\[ \text{for all }y \in X \setminus B \text{ and for all }x \in B\text{ we have }(y,x) \in U.\]
\end{lemma}

\begin{proof}
Suppose $C$ is nonempty, and that the claim does not hold.  In other words assume  that  there exists a $y \in X \setminus C$ and $x \in C$ such that $(y,x) \notin U$.  Now $y \notin C$ so there exists a $z \in X$ such that $(y,z) \in U$. Notice that we may have $y=z$ if $y\in X_U$, but then $\alpha_{y} >1$, and $z \neq x$ by assumption. Now we have $(y,z) \in U$, $(y,x) \notin U$, and $(x,y), (x,z) \notin U$ since $x\in C$. Since $(x,y), (y,x) \notin U$, and $(x,z) \notin U$, Lemma~\ref{lem: underlined blocks 13} applied to $U^{c}$ yields $(y,z) \notin U$, which is a contradiction. 

Now suppose $C$ is empty, $B$ is the block defined in Lemma \ref{lem: smallest block}, and the claim does not hold.  In other words, there exists a $y \in X \setminus B$ and $x \in B$ such that $(y,x) \notin U$.  Since $y\notin B$, and $C$ is nonempty there must be an element $z$ such that $(y,z) \in U$.  Now we have $(y,z) \in U$, $(y,x) \notin U$, and $(x,y), (x,z) \notin U$ since $x\in B$.  Therefore, the same argument as above gives a contradiction.
\end{proof}

\begin{proof}[\textbf{Proof of Theorem~\ref{thm: main inv} }]

Theorem \ref{thm: main inv} can be proved using induction on the size of the set $X$. Suppose first that $C \neq \emptyset$. Then, by Lemma~\ref{lem: block structure}, $C \times X = \emptyset$  and $(X \setminus C) \times C \subset U$. Consider $\invup$ and $\majup$ over the rearrangement class $\R(\al')$ of the permutations of the multiset $\{x^{\alpha_{x}} \colon x \in X\setminus C\}$. Inserting the elements from $C$ in all possible ways among the letters of a word $w' \in \R(\al')$ results in a set of words $S(w') \subset \R(\al)$. It is not hard to see that as $w$ ranges over $S(w')$, the difference $\invup w - \invup w'$ ranges over the multiset $\{i_{1}+i_{2}+ \cdots +i_{r} \colon 0 \leq i_{1} \leq i_{2} \leq \cdots \leq i_{l} \leq s\}$ where $r = |\alpha(C)|$ and $s = |\alpha(X\setminus C)|$. The same is true for the difference $\majup w - \majup w'$. This is less obvious but follows from a similar property of the classical major index for words (see e.g.~\cite[Lemma 4.6]{chen2010major} for a proof). Let $U_1 = U \cap ((X\setminus C ) \times (X \setminus C))$ be the restriction of $U$ on $X\setminus C$. For $w \in \R(\al')$, $\invup w = \inv'_{U_{1}} w$ and $\majup w = \maj'_{U_{1}} w$.  So,
\begin{align*}
\sum_{w \in \R(\al)} q^{\invup w} &= {|\alpha| \brack |C|} \binom{|C|}{\al(C)} \sum_{w \in \R(\al')} q^{\inv'_{U_{1}} w}\\
\sum_{w \in \R(\al)} q^{\majup w} &= {|\alpha| \brack |C|} \binom{|C|}{\al(C)}  \sum_{w \in \R(\al')} q^{\maj'_{U_{1}} w}.
\end{align*} Thus if $\invup$ and $\sorup$ are equidistributed on $\R(\al)$ then $\inv'_{U_{1}}$ and $\maj'_{U_{1}}$ are equidistributed on $\R(\al')$. By the induction hypothesis, $U_{1}$ is essentially bipartitional relative to $\al'$ and thus $U$ is essentially bipartitional relative to $\al$ with one more non-underlined block $C$.

In the case when $C = \emptyset$ there is a block $B$ such that $B \times (X \setminus B)$ is empty and $X \times B \subset U$. Then we consider the relation $U_1 = U \cap ((X\setminus B) \times (X \setminus B))$ on $X \setminus B$. Similar reasoning as above yields
\begin{align*}
\sum_{w \in \R(\al)} q^{\invup w} &= {|\alpha| \brack |B|} \binom{|B|}{\al(B)} q^{\binom{|B|}{2}}\sum_{w \in \R(\al')} q^{\inv'_{U_{1}} w}\\
\sum_{w \in \R(\al)} q^{\majup w} &= {|\alpha| \brack |B|} \binom{|B|}{\al(B)} q^{\binom{|B|}{2}} \sum_{w \in \R(\al')} q^{\maj'_{U_{1}} w}.
\end{align*} 
So, $U_{1}$ is essentially bipartitional relative to $\al'$ and thus $U$ is essentially bipartitional relative to $\al$ with one more underlined block $B$.
\end{proof}


\section{Graphical Sorting Index}\label{sec: sor}

In this section we will prove Theorem~\ref{thm: main sor}. The ``if'' part follows from the following proposition and~\eqref{fzgf}, while the ``only if'' part follows from Lemma~\ref{lem: natural rels} and Lemma~\ref{lem: max block}. 

\begin{proposition} \label{lem: gen function}

If the relation $U$ satisfies the properties of Theorem \ref{thm: main sor} and has blocks $B_1, \ldots , B_k$ then

$$ \sum_{w \in \R(\al)} q^{\sorup w}  = {|\alpha| \brack m_1,\ldots, m_k}\prod_{j=1}^k {m_j \choose \alpha(B_j)}.$$
\end{proposition}

\begin{proof}
We will prove the statement using a $\bcode$ for the words in $\R(\al)$ that we define. Let $w=x_1x_2\dots x_l \in \R(\al)$.   $\bcode{w}$ is a pair of two sequences: a sequence of partitions and a sequence of nonnegative integers. Precisely, we define $\bcode{w}$ to be

\[\left((b_{1,1} \geq  b_{1,2} \geq \ldots \geq b_{1, m_1}; b_{2,1} \geq b_{2,2} \geq \dots \geq b_{2, m_2}; \ldots; b_{k,1} \geq b_{k, 2} \geq \ldots \geq b_{k, m_k}), (p_1, p_2, \ldots, p_k)\right)\] where
\begin{itemize}
\item[($1^{\circ}$)] for $i <k$ each partition $b_{i,1} \geq \ldots \geq b_{i, m_i} \geq 0$  each part has size $b_{i,j} \leq m_{i+1} + m_{i+2} + \cdots + m_{k}$, $1 \leq j \leq m_{i}$, while $b_{k,j} =0$ for $1 \leq j \leq m_{k}$,

\item[($2^{\circ}$)] $p_{i} =0 $ if $|B_{i}| =1$ and $1 \leq p_{i} \leq m_{i}$ if $|B_{i}| =2$.
\end{itemize} 

$\bcode w$ is computed as follows.

\begin{itemize}
\item[(1)] Set $j=1$.
\item[(2)] If $B_j = \{y_{1}, y_{2}\}$ has two integers $y_2 >y_1$  then let  $p_{j}=i$ be the position of $y_{2}$ in the subword of $w$  formed by the elements of $B_j$. Otherwise set $p_{j}=0$. 
\item[(3)] Sort the elements of the block $B_{j}$ and form the partition $b_{j,1} \geq \ldots \geq b_{j, m_j} \geq 0$ from the contributions to $\sor{w}$ (listed in  nonincreasing order) by the elements of $B_{j}$. Keep calling the partially sorted word $w$.
\item[(4)] If $j<k$ increase $j$ by 1 and go to step $(2)$. Otherwise stop. 
 \end{itemize}

Consider, for example, the relation $U = \{(5,3),(5,2),(5,1),(4,3),(4,2),(4,1),(3,2),(3,1)\}$ which is bipartitional with blocks $B_1 = \{5,4\}, B_2 = \{3\}, B_3 = \{2,1\}$ and $\beta_1 = \beta_2 = \beta_3 =0$.  Let $w = 42345411 \in \R(2,1,1,3,1)$.  Snce the subword formed by the 4's and the 5 is $4454$, we have $p_{1} = 3$. The steps for sorting the 4's and the 5 are
$$4234541 \overset{+1} \rightarrow 42341415 \overset{+1} \rightarrow 42341145 \overset{+2} \rightarrow 42311445 \overset{+4} \rightarrow 12314445$$ and, therefore, 
the first partition in $\bcode{w}$ is $4 \geq 2 \geq 1 \geq 1$. Then $p_{2} = 0$ and sorting the 3 yields $12134445$, therefore the second partition  is 1. Finally, $p_{3}=2$ and 
$$\bcode{w} = ((4 \geq 2 \geq 1 \geq 1; 1; 0 \geq 0 \geq 0),(3, 0, 2)).$$

Since the parts of the partitions in the $\bcode$ represent contributions to the sorting index, the bound for their size $b_{i,j} \leq m_{i+1} + m_{i+2} + \cdots + m_{k}$ easily follows. Therefore, the $\bcode$ is clearly a map from $\R(\al)$ to the set of pairs of sequences of partitions and integers which satisfy  $(1^{\circ})$ and $(2^{\circ})$, which we claim is a bijection. For describing the inverse, the crucial observation is that for blocks of size 2, $B_{j} = \{y_{1} < y_{2}\}$, the contribution to the sorting index is given by $b_{j, p_{j}}$. Then given $$\left((b_{1,1} \geq  b_{1,2} \geq \ldots \geq b_{1, m_1}; b_{2,1} \geq b_{2,2} \geq \dots \geq b_{2, m_2}; \ldots; b_{k,1} \geq b_{k, 2} \geq \ldots \geq b_{k, m_k}), (p_1, p_2, \ldots, p_k)\right)$$ which satisfies $(1^{\circ})$ and $(2^{\circ})$, the corresponding word $w \in \R(\al)$ is constructed as follows.

\begin{itemize}
\item[(1)] Let $j=k$ and $w$ be the empty word.
\item[(2)] Add to the end of $w$ the elements of $B_{j}$ with their multiplicities, listed in nondecreasing order $x_{j,1}x_{j,2} \cdots x_{j, m_{j}}$. 
\item[(3)] If $|B_{j}| =1$, then for $i =1, \ldots, m_{j}$, swap $x_{j,i}$ with the element of $w$ which is $b_{j,i}$ places to the left of $x_{j,i}$.
\item[(4)] If $B_{j} = \{y_{1}< y_{2}\}$, then let $b_{j,1}^{'} \geq \cdots \geq b_{j,m_{j}-1}^{'}$ be the partition obtained from $b_{j,1} \geq \ldots \geq b_{j, m_j}$ by deleting the part $b_{j, p_{j}}$. Then for $i =1, \ldots, m_{j}-1$, swap $x_{j,i}$ with the element of $w$ which is $b_{j,i}^{'}$ places to the left of $x_{j,i}$. Finally, swap $x_{j,m_{j}} = y_{2}$ with the element  in $w$ which is $b_{j, p_{j}} + m_{j} - p_{j}$ positions to its left. (After this step there are $b_{j, p_{j}}$ elements from $B_{j+1}, \ldots, B_{k}$ and $m_{j} - p_{j}$ elements from $B_{j}$ to the right of $y_{2}$.) 
\item[(5)] If $j>1$ decrease $j$ by 1 and go to step $(2)$. Otherwise stop. 
\end{itemize}

The $\bcode$ is designed so that $\sorup{w} = \sum_{i=1}^{k} \sum_{j=1}^{m_{i}} b_{i,j}$. The bijection described above then yields  the generating function for $\sorup$. Let $p(j, k, n)$ denote the number of partitions of $n$ into at most $k$ parts, with largest part at most $j$. It is known that $\sum_{n\geq0} p(j,k,n)q^{n} = {j+k \brack j}$. The block $B_j$ contributes  $$ {m_j \choose \alpha(B_j)} \sum_{n\geq 0} p(m_{j+1}+ m_{j+2}\cdots+m_{n},m_j,n)q^n = {m_j \choose \alpha(B_j)} {m_{j}+ m_{j+1}\cdots+m_{n}\brack m_{j}}$$
to $\sum_{w \in \R(\al)} q^{\sorup w}$, where the leading binomial coefficient counts the number of possible values of $p_{j}$. Thus we have 
$$\sum_{w \in \R(\al)} q^{\sorup w} = \prod _{j=1}^k {m_j \choose \alpha(B_j)} {m_{j}+ m_{j+1}\cdots+m_{n} \brack m_{j}} =
{|\alpha| \brack m_1,\ldots, m_k}\prod_{j=1}^k {m_j \choose \alpha(B_j)}.$$

\end{proof}
In particular, we get the generating function for the standard sorting index for words.
\begin{corollary}
\[\sum_{w \in \R(\al)} q^{\sor w} = {|\alpha| \brack m_1,\ldots, m_k}.\]
\end{corollary}

Finally, we prove the ``only if'' part of Theorem~\ref{thm: main sor} via the following few lemmas. 

\begin{lemma} \label{lem: natural rels}
If $\sorup$, $\majup$, and $\invup$ are equidistributed over a fixed rearrangement class $\R(\al)$ then the relation $U$ must be a subset of the integer order modulo relations $(x,x)$.
\end{lemma}

\begin{proof}
Suppose $\sorup$, $\majup$, and $\invup$ are equidistributed on $\R(\al)$. By Theorem \ref{thm: main inv},  $U$ must be essentially bipartitional relative to $\alpha$. That means that there are subsets $I, J  \subset \{x \colon \alpha_{x =1}\}$  such that $U' = (U \setminus \{(x,x) \colon x \in I\}) \cup  \{(x,x) \colon x \in J\}$ is bipartitional. Without loss of generality we may assume that $I, J$ are chosen so that $U'$ does not have underlined blocks $\{x\}$ of size 1 such that $\alpha_{x}=1$. We claim that $U'$ is a subset of the natural order. 

First we will show that there are no underlined blocks  in $U'$.  Suppose the contrary. Then there exist elements $x$ and $y$ such that $(x,y), (y,x) \in U'$ ($x \neq y$ or $y$ is a second  copy of the same element with $\alpha_{x }>1$).   Because we have both $(x,y)$ and $(y,x)$ in $U'$ every word $w \in \R(\al)$  has at least one $U'$-inversion.  Therefore the minimum $\invup$ over the rearrangement class $\R(\al)$ is 1.  On the other hand, $\sorup{11\cdots122\cdots2\cdots nn\cdots n} =0$.  This is a contradiction, and thus there are no underlined blocks in $U'$.

Now assume that $U'$ is not a subset of the natural integer order.  Then there exist at least two elements such that $(x,y) \in U'$, but $y>x$ with respect to the natural order.  Let $B_1, B_2, \ldots, B_k$ be the blocks of $U'$.  Now consider the words created by placing the elements of $B_1$ in some order followed by the elements of $B_2$ placed to the right of $B_1$ and continue the process until the elements of $B_k$ in some order are the last elements of the word.  The words of this type will have $\invup$ equal to the number of edges in the graph $(\al, U')$ as defined in the proof of Lemma~\ref{lem: inv < maj}.  Therefore, the maximum $\invup$ is bounded below by the number of edges in $(\al, U')$ (it is in fact equal to the number of edges in $(\al, U')$). In the sorting algorithm, however, elements are only sorted over elements that are smaller than them with respect to the natural order. Therefore $x$ will never jump over $y$, and thus the relation $(x,y)$ will never contribute to the sorting index.  Since each edge of the graph $(\al, U')$ contributes at most 1 to $\sorup$, we conclude that the maximum $\sorup$ on $\R(\al)$ is less than the maximum $\invup$.  This is a contradiction, and $U'$ must be a subset of the natural order.
\end{proof}

The next inequality will be used to prove the remaining part of Theorem~\ref{thm: main sor}.

\begin{lemma} \label{lem: ineq}
For $a, b \in \mathbb{Z}_{\geq 1}$
\[ \sum_{i=0}^{\min\{a,b\}} \binom{a}{i} \leq \binom{a + b}{b}\] and equality holds if and only if $b=1$.
\end{lemma}

\begin{proof}
If $a \leq b$ then using the Vandermonde's Identity we have
\[\sum_{i=0}^{\min\{a,b\}} \binom{a}{i} = \sum_{i=0}^{a} \binom{a}{i} \leq  \sum_{i=0}^{a} \binom{a}{i} \binom{b} {a-i} = \binom{a+b}{b}\] and equality holds if and only if $a=b=1$.
Similarly, if $a > b$ then 
\[\sum_{i=0}^{\min\{a,b\}} \binom{a}{i} = \sum_{i=0}^{b} \binom{a}{i} \leq \sum_{i=0}^{b} \binom{a}{i} \binom{b} {b-i} = \binom{a+b}{b}\]
\end{proof}

\begin{lemma} \label{lem: max block}
Suppose $U$ is a bipartitional relation with blocks $B_{1}, \ldots, B_{k}$, none of which are underlined, such that  $\sorup$, $\majup$, and $\invup$ are equidistributed over $\R(\al)$. Then for every $1 \leq i < k$, $|B_{i}| \leq 2$ and if the equality $|B_{i}| =2$ holds then $\alpha_{\max B_{i}}=1$. 
\end{lemma}

\begin{proof} By Lemma~\ref{lem: natural rels}, the blocks $B_{1}, \ldots, B_{k}$ are consecutive intervals with $n \in B_{1}$ and $1 \in B_{k}$.	If $k=1$ there is nothing to prove, so suppose $k>1$.

Let $\maxi(B_1,\ldots B_k)$ and $\maxs(B_1, \ldots , B_k)$ denote the number of words in $\R(\al)$ that maximize $\invup$ and $\sorup$, respectively. Let $B_{1} = \{s, s+1, \ldots, n\}$, $s \leq n-1$. The words in $\R(\al)$ that maximize $\invup$ are exactly those formed by a permutation of the elements of $B_{1}$ (with their multiplicities) followed by a permutation of the elements from $B_{2}$, etc.  So, $\maxi(B_1,\ldots, B_k) = \prod_{i=1}^{k} \binom{m_{i}}{\alpha(B_{i})}$. 

On the other hand, if $w \in \R(\al)$ maximizes $\sorup$ then after sorting the $n$'s, one obtains a word $w' \in \R(\al')$ that maximizes $\sorup$ for $\al' = (\alpha_{1}, \ldots, \alpha_{n-1})$. The map $w \to w'$ is not one-to-one. One can write $w' = u v$ where $u$ is the longest prefix of $w'$ formed by elements of $B_{1}$. Then the number of words $w$ that yield $w'$ is at most $\sum_{i=0}^{\min\{|u|, \alpha_{n}\}} \binom{|u|}{i}$.  Namely, such a $w$ can be obtained by appending the $\alpha_{n}$ copies of $n$ to $w'$ and then swapping the leftmost $i$ copies of $n$ with $i$ letters from $u$ and the remaining  $\alpha_{n} - i$ copies of $n$ with the first $\alpha_{n} - i$ letters of $v$.

Since, by Lemma~\ref{lem: ineq}, \[\sum_{i=0}^{\min\{|u|, \alpha_{n}\}} \binom{|u|}{i} \leq \binom{|u| + \alpha_{n}}{\alpha_{n}} \leq \binom{\alpha_{n}+ \alpha_{n-1}+ \cdots + \alpha_{s}}{\alpha_{n}}\]
with equality when $\alpha_{n}=1$, we have
\[ \maxs(B_1, \ldots , B_k) \leq \binom{\alpha_{n}+ \alpha_{n-1}+ \cdots + \alpha_{s}}{\alpha_{n}} \maxs(B_1 \setminus \{n\}, \ldots , B_k),\] where $\maxs(B_1 \setminus \{n\}, \ldots , B_k)$ is the number of words in $\R(\al')$ that maximize $\sorup$. So, inductively, we get

\[ \maxs(B_1, \ldots , B_k) \leq \binom{\alpha_{n}+ \alpha_{n-1}+ \cdots + \alpha_{s}}{\alpha_{s}, \ldots, \alpha_{n-1}, \alpha_{n}} \maxs(B_2, \ldots , B_k) \leq \prod_{i=1}^{k} \binom{m_{i}}{\alpha(B_{i})} = \maxi(B_1,\ldots, B_k).\] Since we have equalities everywhere,  $\alpha_{n}=1$. We also get that $\maxs(B_1 \setminus \{n\}, \ldots , B_k) = \maxi(B_1 \setminus \{n\}, \ldots , B_k)$ and by the same argument, $\alpha_{n} = \alpha_{n-1} = \cdots =\alpha_{s+1}=1$. 

Now consider a permutation $p$ of the multiset $\{1^{\alpha_{1}}, 2^{\alpha_{2}}, \ldots, s-1^{\alpha_{s-1}}\}$ which maximizes $\sorup$. By appending $\alpha_{s}$ copies of $s$ to $p$ and then swapping them with the first $\alpha_{s}$ letters of $p$ we get the word \[\underbrace{ss\cdots s}_{\alpha_{s}} p'.\] One can readily see that the word \[w' = (n-1)\underbrace{ss\cdots s}_{\alpha_{s}-1} p' s (s+1) (s+2) \cdots (n-2) \in \R(\al')\] maximizes $\sorup$ over $\R(\al')$. Also, there are exactly $\alpha_{s} +1$ words $w$ in $\R(\al)$ that maximize $\sorup$ which can be obtained from $w'$, namely, \\ 
\begin{align*}
&n\underbrace{ss\cdots s}_{\alpha_{s}-1} p' s (s+1) (s+2) \cdots (n-2) (n-1), \\
&(n-1)n\underbrace{ss\cdots s}_{\alpha_{s}-2} p' s (s+1) (s+2) \cdots (n-2)s, \\
&(n-1)sn\underbrace{ss\cdots s}_{\alpha_{s}-3} p' s (s+1) (s+2) \cdots (n-2)s, \\
 &\ldots \\
&(n-1)\underbrace{ss\cdots s}_{\alpha_{s}-2} n p' s (s+1) (s+2) \cdots (n-2) s, \\
&(n-1)\underbrace{ss\cdots s}_{\alpha_{s}-1} n p'' s (s+1) (s+2) \cdots (n-2) a, 
\end{align*}where $a$ is the first letter of $p'$. However, as we saw above, if $\sorup$ and $\invup$ are equidistributed on $\R(\al)$, each word $w'$ corresponds to exactly $\binom{\alpha_{n}+ \alpha_{n-1}+ \cdots + \alpha_{s}}{\alpha_{n}}$ words $w$. So,  
\[\binom{\alpha_{n}+ \alpha_{n-1}+ \cdots + \alpha_{s}}{\alpha_{n}} = \alpha_{s} +1\] and therefore $s=n-1$.

This proves that either $B_{1} = \{n-1, n\}$ with $\alpha_{n} =1$ or $B_{1} = \{n\}$. Since the block is of this form, reasoning as in the proof of Proposition~\ref{lem: gen function} one can see that 
\[ \sum_{w \in \R(\al)} q^{\sorup w} =   {m_1 \choose \alpha(B_1)} {m_{1}+ m_{2}\cdots+m_{n}\brack m_{j}}  \sum_{w \in \R(\al'')} q^{\sorup w},\] where $\R(\al'')$ is the set of all permutations of the elements of $B_{2}, \ldots, B_{k}$ with the multiplicities given by $\al$. Since 
\[ \sum_{w \in \R(\al)} q^{\invup w} =   {m_1 \choose \alpha(B_1)} {m_{1}+ m_{2}\cdots+m_{n}\brack m_{j}}  \sum_{w \in \R(\al'')} q^{\invup w},\] we conclude that $\sorup$ and $\invup$ are equdistributed on $\R(\al'')$ and inductively, we get that each of the remaining blocks $B_{2}, \ldots, B_{k-1}$ has either size 1 or size 2 with the multiplicity of the largest element being 1.
\end{proof}

This completes the proof of Theorem~\ref{thm: main sor}.

\end{document}